\documentclass[11pt]{amsart}

\usepackage{amsmath}
\usepackage{amsfonts}
\usepackage{amssymb}

\usepackage{hyperref}
\hypersetup{colorlinks,linkcolor={red},citecolor={blue},urlcolor={blue}}

\newtheorem{theorem}{Theorem}[section]

\newtheorem{corollary}[theorem]{Corollary}

\theoremstyle{definition}
\newtheorem{definition}[theorem]{Definition}

\theoremstyle{remark}

\numberwithin{equation}{section}

\newcommand{\CC}{\mathbb C}
\newcommand{\HH}{\mathbb H}

\newcommand{\NN}{\mathbb N}

\newcommand{\QQ}{\mathbb Q}
\newcommand{\RR}{\mathbb R}
\newcommand{\ZZ}{\mathbb Z}
\newcommand{\SL}{\mathop{\mathrm {SL}}\nolimits}

\def\A-Lift{\operatorname{A-Lift}}
\def\Borch{\operatorname{Borch}}

\newcommand{\m}{\operatorname{mod}}

\begin{document}

\title[Graded rings of integral Jacobi forms]
{Graded rings of integral Jacobi forms}

\author{Valery Gritsenko}

\address{Laboratoire Paul Painlev\'{e}, Universit\'{e} de Lille, 59655 Villeneuve d'Ascq Cedex, France and National Research University ``Higher School of Economics'', Russian Federation}

\email{Valery.Gritsenko@univ-lille.fr}

\author{Haowu Wang}

\address{Max-Planck-Institut f\"{u}r Mathematik, Vivatsgasse 7, 53111 Bonn, Germany}

\email{haowu.wangmath@gmail.com}


\date{\today}


\begin{abstract}
We determine the structure of the bigraded ring of weak Jacobi forms with 
integral Fourier coefficients. This ring is the target ring of a map 
generalising the Witten and elliptic genera and a partition function of 
$(0,2)$-model in string theory. We also determine the structure of the
graded ring of all weakly holomorphic Jacobi forms of weight zero and integral index with integral Fourier coefficients. 
These forms are the data for Borcherds products for the Siegel paramodular 
groups.  
\end{abstract}
\maketitle

\section{Introduction: Jacobi modular forms as partition functions}

Jacobi modular forms of different types  appear in number theory,
geometry, physics and the theory of affine Lie algebras as special 
partition functions.
For example, the elliptic genus of any complex variety with trivial first 
Chern class is a weak Jacobi form  of weight $0$ with integral Fourier 
coefficients (see \cite{G99a}, \cite{KYY}, \cite{T00}).
This Jacobi form generates an automorphic Borcherds product determining 
the second quantised elliptic genus of Calabi-Yau manifolds 
(see \cite{DMVV97}, \cite{G99a}, \cite{GN95}) and  partition 
functions of some string models defined on Calabi-Yau manifolds with  
elliptic fibrations. 
More generally one can define a partition function for a vector 
bundle over a complex manifold. This function is 
also  a meromorphic function of Jacobi type
(see \cite{G99b}).

Let $M=M_d$ be an (almost) complex compact manifold of (complex) 
dimension $d$ and  $E=E_r$ be a holomorphic vector bundle of rank $r$ 
over $M$.
We set  
$q=\exp(2\pi i \tau)$ and  $\zeta=\exp(2\pi i z)$, where $\tau\in \HH$, $z\in \CC$.  Let $K(M)$ be the class group of vector bundles on the variety $M$. One can define a formal  series $\mathbf{E}_{q,\zeta}\in K(M)[[q, \zeta^{\pm1}]]$ 
$$
\mathbf{E}_{q,\zeta}=  \bigotimes_{n= 0}^{\infty}
{\bigwedge}_{-\zeta^{-1}q^{n}}E^*
\otimes 
 \bigotimes_{n= 1}^{\infty}{\bigwedge}_{-\zeta q^n}  E\otimes 
 \bigotimes_{n= 1}^{\infty} S_{q^n} T_M^*
\otimes 
 \bigotimes_{n= 1}^{\infty} S_{q^n} T_M, 
$$
where $E^*$ is dual to $E$,
$T_M$ denotes the holomorphic tangent bundle of $M$, $\wedge^k$ is the $k^\text{th}$ exterior power, $S^k$ is the $k^\text{th}$ symmetric product, and
$$ 
{\bigwedge}_x E=\sum_{k\ge 0}  (\wedge^k E) x^k, \quad S_x E= \sum_{k\ge 0} (S^k E) x^k.
$$
\smallskip

{\it The modified Witten genus} (see \cite{G99b}) 
of $E_r\to M_d$ is a function of two variables
$\tau\in \HH$ and $z\in \CC$
$$
\chi(M_d, E_r;\tau,z)= q^{(r-d)/12}\zeta^{r/2}\int_{M_d}  
\exp\biggl(\frac 1{2}\bigl(c_1(E_r)-c_1(T_{M_d})\bigr)\biggr)\cdot
$$
$$
\exp\biggl(\bigl(p_1(E_r)-p_1(T_{M_d})\bigr)\cdot G_2(\tau)\biggr)
\exp\biggl(-\frac  {c_1(E_r)}{2\pi i} 
\frac {\vartheta_z}{\vartheta}(\tau,z)\biggr)
\hbox{ch}(\mathbf{E}_{q,\zeta})\hbox {td}(T_{M_d}),
$$
where 
$c_1(E_r)$ and $p_1(E_r)$ are the first Chern and Pontryagin classes
of $E_r$, 
$\hbox {td\,}$ is the Todd class, $\hbox{ch}(\mathbf{E}_{q,\zeta})$
is the Chern  character applied to each  coefficient
of the formal power series, and 
$\int_M$ denotes the evaluation of 
the top degree differential form
on the fundamental cycle of the manifold.
In the definition of $\chi(M_d, E_r;\tau,z)$ 
we use the odd Jacobi theta-series of level two 
(see \cite{Mum83})
\begin{equation}\label{theta}
{\vartheta(\tau ,z)=
q^{1/8}(\zeta^{1/2}-\zeta^{-1/2})\prod_{n\ge 1}\,(1-q^{n}\zeta)(1-q^n 
\zeta^{-1})(1-q^n)},
\end{equation}
where
 $\vartheta_z(\tau,z)=
{\dfrac{\partial \vartheta(\tau, z)}{\partial z}}$
and $ G_2(\tau)=-\frac 1{24}+\sum_{n=1}^{\infty}\sigma_1(n)\,q^n$ is the 
quasi-modular Eisenstein series of weight $2$. 

\smallskip

{\bf Examples.} {\bf 1.} Let $E=0$ be the trivial vector bundle of rank $0$.
Let us  assume that $M$ admits a spin structure and $p_1(M)=0$.
In this case, the partition function is equal to the Witten genus (see 
\cite{Wit88}), up to division by a power of the Dedekind $\eta$-function,
$$
q^{d/12}\chi(M, 0;\tau,z)=
\frac{\hbox{Witten genus\,}(M, \tau)}{\eta(\tau)^{{2d}}}.
$$
It is known that the Witten genus is a $\SL_2(\ZZ)$-modular
form of weight $2d$.

{\bf 2.} Let  $M$ be a complex compact manifold, $E=T_M$,
and $c_1(T_M)=0$. Then this partition function coincides 
with the elliptic genus of $M_d$, i.e.  the  holomorphic Euler  
characteristic of $\mathbf{E}_{q,\zeta}$ (\cite{BL00}, \cite{G99a}, \cite{KYY})
$$
\chi(M, T_M;\tau,z)=\phi_{ell}(M; \tau,z)=
\zeta^{d/2}\int_M \hbox{ch}(\mathbf{E}_{q,\zeta})\hbox{td}(T_M).
$$
It is known that the last function is a weak Jacobi form of weight $0$ and 
index $d/2$ (see \S \ref{Sec:2.1} below).
The target ring for the elliptic genus is the graded ring $J_{0,*}^{w, \ZZ}$ 
of all weak Jacobi 
forms of weight zero and integral index with integral Fourier coefficients.
The ring $J_{0,*}^{w, \ZZ}$ has four generators over $\ZZ$ with one torsion 
relation between them (see \cite{G99a}).

The function $\chi(M_d, E_r;\tau,z)$ may have a pole only at $z=0$. 
It is holomorphic if $r\ge d$ or 
$c_1(E)=0$ (over $\RR$). If $r<d$ then the modified function
$\eta(\tau)^{r-d}\vartheta(\tau,z)^{d-r}\chi (M_d, E_r, \tau, z)$ is a weak 
Jacobi form of weight $0$ and index $d/2$. For example, 
if $c_1(E_2)=\hat A(M_6)=0$  and $p_1(E_2)=p_1(T_{M_6})$, then 
$$
\chi(M_{6}, E_2; \tau, z)=-\hat A(M_6,E_2) E_{4,1}(\tau,z)\eta(\tau)^{-8},
$$
where $E_{4,1}(\tau,z)$ is a Jacobi--Eisenstein series (see \S \ref{Sec:3.2} below), and $\hat A(M)$ and $\hat A(M, E)$ are  the  $\hat A$-genus and the twisted $\hat A$-genus of the variety $M$ (see \cite{G99b}).

In this paper we describe the structure of the bigraded ring 
$J_{*,*}^{w, \ZZ}$ 
of all weak Jacobi forms  with integral Fourier coefficients, which is the ring that contains 
the partition functions $\chi(M_d, E_r;\tau,z)$.
We show that this ring has $14$ generators satisfying $10$ relations 
(see Theorem \ref{th:Gri2}).

In \S \ref{Sec:4} we also determine the structure of the graded ring of weakly 
holomorphic Jacobi forms of weight $0$ with integral Fourier coefficients 
(see Theorem \ref{th:weaklyJF}). 
These functions are the data for automorphic Borcherds products for the 
Siegel paramodular groups (see \cite{GN98}). This ring has $8$ generators of 
three different types: the modular invariant $j(\tau)$,  four generators $
\phi_{0,1},\dots, \phi_{0,4}$ of the elliptic genus type, and  three 
generators related to Jacobi--Eisenstein series that appear in formulas for 
partition functions of  type $\chi(M_{6}, E_2; \tau, z)$. 
All eight generators produce remarkable Borcherds products.
The normalised modular invariant $J(\tau)=j(\tau)-744$ defines the 
monstrous Lie algebra (see \cite{Bor92}).
Other generators determine the main Lorentzian Kac--Moody algebras in the 
Gritsenko--Nikulin classification (see \cite{GN98}).
The first generator $\phi_{0,1}$ and its Kac--Moody algebra  correspond 
to the second quantised elliptic genus of K3  surfaces
mentioned above. The generators $\phi_{0,2}$, $\phi_{0,3}$, $\phi_{0,4}$
take part in the formulas for the second quantised elliptic genus
of Calabi--Yau manifolds of dimensions $4$, $6$ and $8$.  
We suppose that the last three generators in Theorem \ref{th:weaklyJF} 
produce similar objects for 
vector bundles of some Calabi--Yau varieties.

\section{The basic Jacobi modular forms and Jacobi theta-series}
In this section we introduce the main facts related to Jacobi forms that will be useful for subsequent constructions.
We refer to the book \cite{EZ85} for the Jacobi forms of integral weights 
and indices and \cite{GN98} for the case of half-integral indices and 
weights.
\subsection{Definition of Jacobi forms}\label{Sec:2.1}
The Jacobi group $\Gamma^J(\ZZ)$ is the semidirect product of $\SL_2(\ZZ)$ 
with the integral Heisenberg group $H(\ZZ)$ which is the central extension of 
$\ZZ\times\ZZ$. The  natural model of  $\Gamma^J(\ZZ)$ 
is the quotient by $\{ \pm I\}$ of the integral maximal parabolic subgroup of 
the Siegel modular group of genus $2$ fixing an isotropic line. We refer to \cite[\S 1]{GN98} for these definitions.
Let $\chi : \Gamma^J(\ZZ) \rightarrow \CC^*$ be a character (or a multiplier 
system) of finite order. By \cite[Proposition 2.1]{CG13}, its restriction to 
$\SL_2(\ZZ)$ is a power $\upsilon_{\eta}^D$ of the multiplier system of the 
Dedekind $\eta$-function and for some $t\in \{0,1\}$ we have, for every $A\in \SL_2(\ZZ)$, $[x,y;r]\in H(\ZZ)$,
\begin{equation*}
\chi ( A \cdot [x,y;r])= 
\upsilon_{\eta}^D (A) \cdot \upsilon_{H}^t([x,y;r]), 
\quad \upsilon_{H}([x,y;r])=(-1)^{x+y+xy+r},
\end{equation*}
where $\upsilon_{H}$ is the unique binary character of $H(\ZZ)$.

\begin{definition}\label{def:JF}
Let $k\in \frac{1}{2}\ZZ$ and $t\in \frac{1}{2}\NN$. A holomorphic function 
$\phi : \HH\times \CC\to \CC$ is called a {\it weakly holomorphic} Jacobi 
form of weight $k$ and index $t$ with a character (or a multiplier system) 
of finite order $\chi: \Gamma^J(\ZZ) \rightarrow \CC^*$ if it satisfies  
\begin{align*}
\phi\left(\frac{a\tau+b}{c\tau+d},\frac{z}{c\tau+d}\right)&
=\chi(A)(c\tau+d)^ke^{2\pi i t \frac{cz^2}{c\tau+d}}\phi(\tau,z),\\
\phi(\tau,z+x\tau+y)&=\chi([x,y;0])e^{-2\pi i t (x^2\tau+2xz)}\phi(\tau,z),
\end{align*}
for $A=\left(\begin{smallmatrix}
a & b \\ 
c & d \end{smallmatrix}\right)\in \SL_2(\ZZ)$, 
$x,y\in \ZZ$
and has a Fourier expansion of the form
$$ \phi(\tau,z)=\sum_{\substack{ n\geq n_0, n\equiv \frac{D}{24} \m \ZZ\\ l\in \frac{1}{2}\ZZ}} f(n,l)\exp(2\pi i(n\tau+lz)),$$
where $n_0\in\ZZ$ is a constant and $\chi\lvert_{\SL_2(\ZZ)} = \upsilon_{\eta}^D$ with $0\leq D < 24$. 

If $\phi$ satisfies the condition
$(f(n,l) \neq 0 \Longrightarrow n \geq 0) $
then it is called a {\it weak} Jacobi form.
If $\phi$ further satisfies the condition
$(f(n,l) \neq 0 \Longrightarrow 4nt - l^2 \geq 0) $
then it is called a {\it holomorphic} Jacobi form. If $\phi$ further satisfies the stronger condition
$(f(n,l) \neq 0 \Longrightarrow 4nt - l^2 > 0)$
then it is called a Jacobi {\it cusp} form.
We denote by 
$$
J^{!}_{k,t}(\chi)\supseteq J_{k,t}^w(\chi) \supseteq
J_{k,t}(\chi) \supseteq J^{cusp}_{k,t}(\chi)
$$ 
the vector spaces of weakly holomorphic Jacobi forms of weight $k$ and index $t$ with character $\chi$ and the corresponding spaces of weak, holomorphic
and cusp Jacobi forms.
If the character is trivial, we write $J_{k,t}=J_{k,t}(1)$ for short.
\end{definition}

From the definition, it is easy to see that
$$
J^{w}_{k,t}(\upsilon_{\eta}^D\cdot\upsilon_{H}^{2t})=\eta^D\cdot J^{w}_{k-\frac{D}{2},t}(\upsilon_{H}^{2t}),
\quad\text{and}\quad k-\frac{D}{2}\in \ZZ. 
$$

\subsection{Jacobi theta functions}\label{Sec:2.2}
We recall some standard facts about Jacobi theta functions which play an important role in our constructions.
Let $q=e^{2\pi i \tau}$ and $\zeta=e^{2\pi i z}$, where $\tau \in \HH$ and $z\in \CC$. The main Jacobi theta-series is defined in 
\eqref{theta}
$$
\vartheta(\tau, z)=q^{\frac{1}{8}}\zeta^{\frac{1}{2}}
\sum_{n\in \ZZ}(-1)^nq^{\frac{n(n+1)}{2}}\zeta^{n}
\in J_{\frac{1}{2},\frac{1}{2}}(\upsilon_{\eta}^3\cdot\upsilon_{H}).
$$
We note that $\vartheta(\tau, -z)=-\vartheta(\tau, z)$ and
$\vartheta(\tau, z)=-i\vartheta_{11}(z,\tau)$ in the notation of 
\cite[Chapter 1]{Mum83}. Its set of simple zeros is $\ZZ\tau+\ZZ$. There are three other Jacobi theta functions of level two: 
$$
\vartheta_{00}(\tau,z)=
\sum_{n\in \ZZ}q^{\frac{n^2}{2}}\zeta^n,\quad
\vartheta_{01}(\tau,z)=\sum_{n\in \ZZ}(-1)^nq^{\frac{n^2}{2}}\zeta^n,
$$
$$
\vartheta_{10}(\tau,z)=q^{\frac{1}{8}}\zeta^{\frac{1}{2}}
\sum_{n\in \ZZ}q^{\frac{n(n+1)}{2}}\zeta^{n}.
$$
We can consider them as functions conjugate (in the sense of Jacobi group)
to $\vartheta(\tau, z)$ (see \cite{GW18}).
We also define the corresponding theta-constants
$$ 
\theta_{ab}(\tau)=\vartheta_{ab}(\tau,0)
\quad\text{and}\quad
\xi_{ab}(\tau,z)=\vartheta_{ab}(\tau,z)/\vartheta_{ab}(\tau,0),
$$ 
for 
$ab=00$, $01$, $10$. From their infinite product expansions, we see that 
$\xi_{00}$, $\xi_{01}$ and  $2\xi_{10}$ have integral Fourier coefficients.

\subsection{The bigraded ring of  Jacobi forms over $\CC$}\label{Sec:2.3}
For convenience, we first introduce some notations of basic modular forms 
and Jacobi forms. Let $E_k$ denote the Eisenstein series of weight $k$ on $
\SL_2(\ZZ)$. Let $\Delta=\eta^{24}$ denote the cusp form of weight $12$ on $
\SL_2(\ZZ)$. The function $E_{k,m}$ denotes the Jacobi-Eisenstein series of 
weight $k$ and index $m$ introduced in \cite{EZ85}. Throughout the paper, 
the symbol $\zeta^{\pm m}$ denotes the sum $\zeta^{m}+\zeta^{-m}$.

The structure of the ring of weak Jacobi forms over $\CC$ is well-known and 
rather simple. 
Firstly, \cite[Theorem 9.4]{EZ85} tells that the bigraded ring 
of weak Jacobi forms of even weight and integral index
$J_{2*,*}^w=
\oplus_{k\in \ZZ, t\in \NN}\,J_{2k,t}^w$  
is a polynomial algebra in two generators 
$\phi_{0,1}$ and $\phi_{-2,1}$ over the graded ring of $\SL_2(\ZZ)$-modular 
forms.
There are many different constructions for these generators (see
\cite{EZ85}, \cite{G99a}, \cite{GN98})
\begin{equation}
\begin{split}
\phi_{-2,1}(\tau,z)&=\frac{\vartheta^2(\tau,z)}{\eta^6(\tau)}=
\frac{E_6(\tau)E_{4,1}(\tau, z)-E_4(\tau)E_{6,1}(\tau, z)}
{144\Delta(\tau)}\\
&=\zeta^{\pm 1}-2+ q(-2\zeta^{\pm 2} +8\zeta^{\pm 1}-12)+O(q^2)
\in J_{-2,1}^{w},
\end{split}
\end{equation}
\begin{equation}
\begin{split}
\phi_{0,1}(\tau,z)&=-\frac{3}{\pi^2}\wp(\tau,z)\phi_{-2,1}(\tau,z)=4(\xi_{00}^2+\xi_{01}
^2+\xi_{10}^2)(\tau,z)\\
&=\zeta^{\pm 1}+10+ q(10\zeta^{\pm 2} - 64\zeta^{\pm 1}+108)+O(q^2)\in 
J_{0,1}^{w},
\end{split}
\end{equation}
where $\wp(\tau,z)\in J_{2,0}^{mer}$ is the Weierstrass function.

We note that $2\phi_{0,1}(\tau,z)$ is equal to the elliptic genus of $K3$ 
surfaces. The function $\phi_{0,1}$ itself is equal to the elliptic genus of 
Enriques surfaces.
  
For $k\in \ZZ$ and $t\in \NN$, we have (see \cite{G99a}):
\begin{align}\label{eq:structure1}
&J_{2k+1,t}^w=\phi_{-1,2}\cdot J_{2k+2,t-2}^w,& &\phi_{-1,2}(\tau,z)=
\frac{\vartheta(\tau,2z)}{\eta^3(\tau)},&
\end{align}
\begin{align}\label{eq:structure2}
&J_{2k,t+\frac{1}{2}}^w=\phi_{0,\frac{3}{2}}\cdot J_{2k,t-1}^w,& &J_{2k+1,t+
\frac{1}{2}}^w=\phi_{-1,\frac{1}{2}}\cdot J_{2k+2,t}^w,&
\end{align}
where 
\begin{align}
&\phi_{0,\frac{3}{2}}(\tau,z)=\frac{\vartheta(\tau,2z)}{\vartheta(\tau,z)},& 
&\phi_{-1,\frac{1}{2}}(\tau,z)=\frac{\vartheta(\tau,z)}{\eta^3(\tau)}.& 
\end{align}
Since the functions $\phi_{-1,2}$, $\phi_{0,\frac{3}{2}}$ and $
\phi_{-1,\frac{1}{2}}$ have  infinite product expansions, the relations 
(\ref{eq:structure1}) and 
(\ref{eq:structure2}) hold for the corresponding spaces of Jacobi forms with 
integral Fourier coefficients. 

We note that the function 
$\frac{e(M)}2\phi_{0,\frac{3}{2}}$ is equal to the elliptic genus of 
Calabi--Yau $3$-folds with Euler number $e(M)$. The function
$\frac{e(M)}{24}\phi_{0,1}\phi_{0,\frac{3}{2}}$ coincides with the elliptic 
genus of Calabi--Yau $5$-folds. It follows that $e(M)\equiv 0\mod 24$
for any Calabi--Yau $5$-folds (see \cite{G99a}).

\section{$\ZZ$-structure of different graded rings of weak Jacobi forms}

\subsection{The graded ring of integral Jacobi forms of weight $0$}\label{Sec:3.1}
We denote by $J_{k,m}^{w,\ZZ}$ the $\ZZ$-module of all weak Jacobi forms of 
weight $k$ and index $m$ with integral Fourier coefficients for $k\in\ZZ$, 
$m\in\NN$.
The main theorem of the paper is Theorem \ref{th:Gri2}.
To prove it we have to construct generators of $J_{*,*}^{w,\ZZ}$. 
Following \cite{G99a},  we introduce the graded ring
$$
J_{0,*}^{w,\ZZ}=\bigoplus_{m\in\NN} J_{0,m}^{w,\ZZ}
$$
and its ideal 
$$
J_{0,*}^{w,\ZZ}(q)=\left\{ \phi \in J_{0,*}^{w,\ZZ} : \phi(\tau,z)=\sum_{n\geq 1, l\in\ZZ}a(n,l)q^n\zeta^l \right\}
$$
of the Jacobi forms without $q^0$-term in their Fourier expansions. 

It is clear that $\phi_{-2,1} \in J_{-2,1}^{w,\ZZ}$ because it has an integral infinite product expansion due to the  Jacobi triple product.
Since $\xi_{00}$, $\xi_{01}$, $2\xi_{10}$ have integral Fourier coefficients, we have $\phi_{0,1} \in J_{0,1}^{w,\ZZ}$ (as the elliptic genus of $K3$ surfaces). 
There are three other integral basic Jacobi forms of index $2$, $3$ and $4$
which play a crucial role in the classification of Lorentzian Kac--Moody 
algebras of hyperbolic rank $3$ (see \cite[Example 2.3 and Lemma 2.5]
{GN98}). We have 
\begin{equation}
\begin{split}
\phi_{0,2}(\tau,z)&= 2(\xi_{00}+\xi_{01}+\xi_{10})(\tau,2z)\\
&=\zeta^{\pm 1} + 4  + q(\zeta^{\pm 3}-8\zeta^{\pm 2}-\zeta^{\pm 1}+16)+O(q^2)\in J_{0,2}^{w,\ZZ},
\end{split}
\end{equation}

\begin{equation}
\begin{split}
\phi_{0,3}(\tau,z)&= \frac{\vartheta(\tau,2z)^2}{\vartheta(\tau,z)^2}\\
&=\zeta^{\pm 1} + 2  + q(-2\zeta^{\pm 3}-2\zeta^{\pm 2}+2\zeta^{\pm 1}+4)+O(q^2)\in J_{0,3}^{w,\ZZ},
\end{split}
\end{equation}

\begin{equation}
\begin{split}
\phi_{0,4}(\tau,z)&= \frac{\vartheta(\tau,3z)}{\vartheta(\tau,z)}\\
&=\zeta^{\pm 1} + 1 -q(\zeta^{\pm 4}+\zeta^{\pm 3}
-\zeta^{\pm 1}-2)+O(q^2)\in J_{0,4}^{w,\ZZ}.
\end{split}
\end{equation}
Let us remark that the formula for $\phi_{0,2}$ above is different from 
the five formulas for this function from \cite{G99a} and \cite{GN98}.
In this form, it is easier to see that $\phi_{0,2}$ has integral Fourier 
coefficients.

The following  \cite[Theorem 1.9]{G99a} describes a structure of $J^{w,\ZZ}_{0,*}$.

\begin{theorem}\label{th:Gri1}
\begin{enumerate}
\item[(a)] Let $m$ be a positive integer. The module 
$$
J_{0,m}^{w,\ZZ} /J_{0,m}^{w,\ZZ}(q)=\ZZ [\psi_{0,m}^{(1)},..., \psi_{0,m}^{(m)}]
$$
is a free $\ZZ$-module of rank $m$. Moreover, there is a basis consisting of $\psi_{0,m}^{(n)}$,  $1\leq n \leq m$, with the following $q^0$-term
\begin{align*}
&[\psi_{0,m}^{(1)}]_{q^0}=\frac{1}{(12,m)}(m\zeta+(12-2m)+m\zeta^{-1}),\\
&[\psi_{0,m}^{(2)}]_{q^0}=\zeta^2-4\zeta+6-4\zeta^{-1}+\zeta^{-2},\\
&[\psi_{0,m}^{(n)}]_{q^0}=\zeta^{\pm n} +\sum_{0\leq i<n}c(i)\zeta^{\pm i}, \; c(i)\in \ZZ, \quad (3\leq n \leq m),
\end{align*}
where $(12,m)$ is the greatest common divisor of $12$ and $m$.

\item[(b)] The ideal $J^{w,\ZZ}_{0,*}(q)$ is principal. It is generated by 
\begin{equation*}
\xi_{0,6}(\tau, z) = \frac{\vartheta^{12}(\tau,z)}{\eta^{12}(\tau)}=q(\zeta^{\frac{1}{2}}-\zeta^{-\frac{1}{2}})^{12}+O(q^2)\in J_{0,6}^{w,\ZZ}.
\end{equation*}

\item[(c)] The graded ring $J^{w,\ZZ}_{0,*}$ is finitely generated over $\ZZ$. More precisely,
$$ J^{w,\ZZ}_{0,*}= \ZZ [\phi_{0,1},\phi_{0,2},\phi_{0,3},\phi_{0,4}]. $$
The functions $\phi_{0,1}$, $\phi_{0,2}$, $\phi_{0,3}$ are algebraically independent over $\CC$ and 
$4\phi_{0,4}=\phi_{0,1}\phi_{0,3}-\phi_{0,2}^2$. Moreover, we have the following formula for the discriminant
$
\xi_{0,6}=-\phi_{0,1}^2 \phi_{0,4}+9\phi_{0,1}\phi_{0,2}\phi_{0,3}-8\phi_{0,2}^3-27
\phi_{0,3}^2.
$
\end{enumerate}
\end{theorem}

\subsection{Weak Jacobi forms with integral coefficients}\label{Sec:3.2}
As a corollary of Theorem \ref{th:Gri1}, we give an efficient method to check
whether the Fourier coefficients of a weak Jacobi form 
are integers or not.

\begin{theorem}\label{th:Jacobi}
Let $\varphi=\sum f(n,l)q^n\zeta^l \in J_{2k,m}^{w}$, where $k\in \ZZ$, $m 
\in \NN$. If the Fourier coefficients of the first $[\frac{m+k}{6}]$ 
$q$-terms are integral, i.e. $f(n,l)\in \ZZ$ for $0\leq n \leq [\frac{m+k}
{6}]$ and $l\in\ZZ$, then all Fourier coefficients of $\varphi$ are 
integral. Here, $[x]$ is the integral part of $x$.
\end{theorem}

\begin{proof}
When $m=0$, the function $\varphi$ is independent of $z$ and, hence,
it is a 
$\SL_2(\ZZ)$-modular form. In this case, the result is clear. 
We next assume that $m\geq 1$.
We first show that the assertion is true when $k=0$. We prove this fact by 
induction on the index $m$ using the following two arguments.

Let $\varphi \in J_{0,m}^{w}$ and assume that all Fourier coefficients of 
$q^0$-term are integral. From Theorem \ref{th:Gri1} (a), there exists a weak 
Jacobi form $\psi\in J_{0,m}^{w,\ZZ}$ with the same $q^0$-term.
Therefore $(\varphi-\psi)/\xi_{0,6}\in J_{0,m-6}^{w}$ 
by Theorem \ref{th:Gri1} (b).

Let $\varphi \in J_{0,m}^{w}$ with $m<6$. If the Fourier coefficients of 
$q^0$-term are integral, then $\varphi \in J_{0,m}^{w,\ZZ}$ according to 
Theorem \ref{th:Gri1} (a), (b).

The general case is reduced  to the case $k=0$.
If $\varphi \in J_{2k,m}^{w}$, then we have that 
$\varphi\cdot \phi_{-2,1}^k \in 
J_{0,m+k}^{w}$, and $\varphi \in J_{2k,m}^{w,\ZZ}$ if and only if 
$\varphi \cdot \phi_{-2,1}^k \in J_{0,m+k}^{w,\ZZ}$.
\end{proof}

The main difference between the $\ZZ$-modules
$J_{0, m}^{w,\ZZ}$ and $J_{k, m}^{w, \ZZ}$ is the fact that  for an even
weight larger than $2$ there is a Jacobi form whose $q^0$-term is equal to $1$.

\begin{corollary}\label{cor:1JF}
For any even weight $2k\ge 4$ and 
index $m\ge 0$ there exists an integral Jacobi form
$$ 
F_{2k,m}=1+O(q)\in J^{w,\ZZ}_{2k,m}.
$$
\end{corollary}
\begin{proof}
If  $m=0$,  one can construct such modular form in $\tau$ 
as a product of the Eisenstein series 
$E_k(\tau)=1-\frac{4k}{B_{2k}}\sum_{n=1}^{\infty}
\sigma_{k-1}(n)q^n$ 
of weight $4$ and $6$ with integral Fourier coefficients.
For $m\geq 1$ one can define  the  Jacobi--Eisenstein series
(see \cite[\S 7]{EZ85})
$$
E_{k,m}(\tau,z)=
\frac 1{2}\sum_{\substack{c,d\in \ZZ\\ (c,d)=1}}
\sum_{l\in \ZZ}
(c\tau+d)^{-k} \exp(2\pi i \bigl(l^2
\tfrac {a\tau+b}{c\tau+d}+2l \tfrac{z}{c\tau+d}-\tfrac{cz^2}{c\tau+d}
\bigr)).
$$
In  general, the Fourier coefficients of $E_{k,m}$ are not integral (see \cite[Lemma 2.3]{GW18}).
In fact, one can use only the series $E_{k,1}$ of index one
because the series $E_{k,m}$ can be expressed in terms of $E_{k,1}$ and the Hecke type operators $V_d$, $U_d$ (see \cite[(13), p.48]{EZ85}).
By Theorem \ref{th:Jacobi}, the Jacobi--Eisenstein series $E_{4,1}$, 
$E_{4,2}$, $E_{4,3}$, $E_{6,1}$, $E_{6,2}$ have integral Fourier coefficients because their $q^0$-terms are all the constant $1$. 
Thus we can define  $F_{2k,m}=E_{2k,m}$ for $(k,m)=(2,1)$, $(2,2)$, $(2,3)$, $(3,1)$, $(3,2)$.  In addition, we put 
\begin{align}\label{F63}
F_{6,3}=\frac{1}{2}(E_{6,1}\phi_{0,2}-E_6\phi_{0,3}).
\end{align}
We calculate its $q^1$-term
$$
F_{6,3}=1+q(2\zeta^{\pm 3}-45\zeta^{\pm 2}-90\zeta^{\pm 1}-238)+O(q^2).
$$
By Theorem \ref{th:Jacobi}, the function $F_{6,3}$ has integral Fourier coefficients. 
So we finish the construction of $F_{2k,m}$ by the following recursion formulas
\begin{align*}
F_{4k,0}&=E_4^k, \qquad F_{4k+2}=E_4^{k-1}E_6,\quad k\geq 1;\\
F_{2k,m}&=F_{2k-4,0}F_{4,m}, \quad k\geq 4, \; m=1,2,3;\\
F_{2k,m}&= F_{2k,m-3}\phi_{0,3}-F_{2k,m-4}\phi_{0,4}, \quad k\geq 2, \; m\geq 4.
\end{align*}
\end{proof} 

Theorem \ref{th:Gri1} gives us also a method to produce a lot of relations 
between weak Jacobi forms, for example,  between 
generators from Corollary \ref{cor:1JF}. 
Using the structure of the divisor of $\phi_{-2,1}$ we get 
\begin{align*}
E_{4}\phi_{0,1}-E_6\phi_{-2,1}&=12 E_{4,1},\qquad
E_{4,1}\phi_{0,1}-E_{6,1}\phi_{-2,1}=12E_{4,2},\\
E_{6}\phi_{0,1}-E_4^2\phi_{-2,1}&=12 E_{6,1},\qquad
E_{6,1}\phi_{0,1}-E_4(\tau)E_{4,1}\phi_{-2,1}=12E_{6,2},\\
E_{4,1}\phi_{0,2}-E_{4}\phi_{0,3}&=2E_{4,3},\qquad
\ \, E_{4,2}\phi_{0,1}-E_{4,1}\phi_{0,2}=6E_{4,3}.
\end{align*}
Using these relations we can easily calculate the first terms in the Fourier 
expansions of the Jacobi--Eisenstein series
\begin{align*}
E_{4,1}(\tau,z)=& 
1+q(\zeta^{\pm 2}+56\zeta^{\pm 1}+126)+O(q^2),\\
E_{4,2}(\tau,z)=&1+q(14\zeta^{\pm 2}+64\zeta^{\pm 1}+84)+O(q^2) ,\\
E_{4,3}(\tau,z)=&1+q(2\zeta^{\pm 3}+27\zeta^{\pm 2}+54\zeta^{\pm 1}+74)+O(q^2),
\end{align*}
(compare with the table at the end of \cite{EZ85}).
In particular,  we obtain that
$$
E_{4,1}(\tau,z)\equiv E_{6,1}(\tau,z) \mod 24, \qquad
E_{4,2}\equiv E_{4,1}^2\mod 12.
$$
Note that $F_{6,3}$ is not a  Jacobi--Eisenstein series because 
the Fourier coefficients of $E_{6,3}$ are not integral.
(See the remark after the proof of Theorem \ref{th:Gri2}.)

We note that there is another way to see that $E_{4,m}$ has integral Fourier 
coefficients for $m=1$, $2$, $3$. For these indices we have 
$$
E_{4,m}(\tau, z)=\vartheta_{E_8}(\tau,zv_m),
\qquad v_m\in E_8, \quad (v_m,v_m)=2m
$$
where 
$$
\vartheta_{E_8}(\tau,\mathfrak{z})
=\sum_{v\in E_8}e^{\pi i (v,v)\tau+2\pi i (v,\mathfrak{z})}
$$
is the Jacobi theta series in $\mathfrak{z}\in E_8\otimes \CC$ for the root 
lattice $E_8$.

\subsection{Siegel modular forms with integral Fourier coefficients}\label{Sec:3.3}
Theorem \ref{th:Jacobi} can be applied to the theory of Siegel modular forms 
of genus $2$. 
In \cite{Igu79}, Igusa determined the structure of the space of Siegel 
modular forms of genus 2 with integral Fourier coefficients. One difficulty 
in his proof is to show that the generators have integral Fourier 
coefficients. However, Theorem \ref{th:Siegel} provides us with an efficient way to do this.
We hope that using the last theorem we can determine the 
integral structure of the graded ring of Siegel modular forms for the 
paramodular groups $\Gamma_2$ and $\Gamma_3$.

\begin{theorem}\label{th:Siegel}
Let 
$$\Phi(\tau_1,z,\tau_2)=\sum_{\substack{n,m\in\NN, l\in\ZZ\\4nm\geq l^2}}
c(n,l,m)q_1^n q_2^m \zeta^l$$
be a Siegel modular form of genus $2$ and weight $2k$, where $q_1=e^{2\pi i 
\tau_1}$, $q_2=e^{2\pi i \tau_2}$, $\zeta=e^{2\pi i z}$. If $c(n,l,m)\in \ZZ
$ for all $m\leq \left[\frac{k+1}{5} \right]$, $n\leq \left[\frac{k+m}
{6}\right]$ and $l^2\leq 4nm$, then $c(n,l,m)$ are all integral, which means 
that $\Phi$ has integral Fourier coefficients. 
\end{theorem}
\begin{proof}
We first write 
$$
\Phi(\tau_1,z,\tau_2)=\sum_{m\geq 0}f_m(\tau_1,z)q_2^m,
$$
where
$$
f_m(\tau_1,z)=
\sum_{\substack{n\in\NN, l\in\ZZ\\4nm\geq l^2}}c(n,l,m)q_1^n \zeta^l.
$$
It is well-known that $f_m(\tau_1,z)$ is a holomorphic Jacobi form of 
weight $2k$ and index $m$.
Note that $c(n,l,m)=c(m,l,n)$. Assume that $f_1, \cdots, f_m$ have integral 
Fourier coefficients. Then the Fourier coefficients of the first $m$ 
$q$-terms of $f_{m+1}$ are all integral. By Theorem \ref{th:Jacobi}, 
if $[\frac{k+m+1}{6}]\leq m$ then $f_{m+1}$ has integral Fourier 
coefficients. Suppose that $\frac{k+1}{5} = x + \frac{y}{5}$ with $x\in \NN$ 
and $0\leq y \leq 4$. Fix $m\geq\left[\frac{k+1}{5} \right]$. Then we can 
write $m=x+t$ with $t\in \NN$. We have
$$
\left[\tfrac{k+m+1}{6}\right]=\left[ \tfrac{(5x+y)+(x+t)}{6}  \right]=x+
\left[\tfrac{y+t}{6}\right] \leq x + t.
$$
Hence, if $f_i$, $0\leq i \leq \left[\frac{k+1}{5} \right]$, have integral 
Fourier coefficients, then $f_j$ also has integral Fourier coefficients for 
any $j>\left[\frac{k+1}{5} \right]$.  The proof is completed by Theorem 
\ref{th:Jacobi}. 
\end{proof}

\noindent{\bf Remark.}
The above result is a very simple application of our results on Jacobi 
modular forms. One can also prove it using the so-called Sturm type 
theorems. This type of results gives congruence criteria which say that if 
the Fourier coefficients of $F(Z)$ from a certain finite set are congruent 
to zero modulo a prime $p$ then all Fourier coefficients of $F(Z)$ are 
congruent to zero (see, for example, \cite{KT19} and references there).

\subsection{The bigraded ring of weak Jacobi forms}\label{Sec:3.4}
We now can state and prove the main theorem.  

\begin{theorem}\label{th:Gri2}
The bigraded ring of all weak Jacobi forms of even weight and integral index 
with integral Fourier coefficients is finitely generated over $\ZZ$. More exactly, we have
$$
J^{w,\ZZ}_{2*,*}=\ZZ [E_4,E_6,\Delta, E_{4,1},E_{4,2},E_{4,3},
E_{6,1},E_{6,2},F_{6,3},\phi_{0,1},\phi_{0,2},\phi_{0,3},\phi_{0,4},\phi_{-2,1}],
$$
where $E_{k,m}$ are Jacobi--Eisenstein series and $F_{6,3}$ is defined in 
\eqref{F63}.
The generators $E_4$, $E_6$, $\phi_{0,1}$ and $\phi_{-2,1}$ are algebraically independent over $\CC$. 
The rest of the generators belong to 
$\QQ[E_4,E_6,\phi_{0,1},\phi_{-2,1}]$. 
More precisely, we have
\begin{align*}
(2^6\cdot 3^3)\Delta=&E_4^3-E_6^2,\\
(2^3\cdot 3)\phi_{0,2}=&\phi_{0,1}^2-E_4\phi_{-2,1}^2,\\
(2^4\cdot 3^3)\phi_{0,3}
=&\phi_{0,1}^3-3E_4\phi_{0,1}\phi_{-2,1}^2+2E_6\phi_{-2,1}^3,\\
4\phi_{0,4}=&\phi_{0,1}\phi_{0,3}-\phi_{0,2}^2,\\
12E_{4,1}=&E_4\phi_{0,1}-E_6\phi_{-2,1},\\
12 E_{6,1}=&E_6\phi_{0,1}-E_4^2\phi_{-2,1},
\end{align*}
\begin{align*}
6E_{4,2}=&E_{4,1}\phi_{0,1}-E_4\phi_{0,2}=
\frac{1}{24}
(E_4\phi_{0,1}^2-2E_6\phi_{0,1}\phi_{-2,1}+E_4^2\phi_{-2,1}^2),\\
6E_{6,2}=&E_{6,1}\phi_{0,1}-E_6\phi_{0,2}=
\frac{1}{24}
(E_6\phi_{0,1}^2-2E_4^2\phi_{0,1}\phi_{-2,1}+E_4E_6\phi_{-2,1}^2),\\
2E_{4,3}=&E_{4,1}\phi_{0,2}-E_4\phi_{0,3}
=\frac{1}{3}(E_{4,2}\phi_{0,1}-E_{4,1}\phi_{0,2})\\
 =&\frac{1}{4}(E_{4,2}\phi_{0,1}-E_4\phi_{0,3}),\\
2F_{6,3}=&E_{6,1}\phi_{0,2}-E_6\phi_{0,3}=\frac{1}{3}(E_{6,2}\phi_{0,1}-E_{6,1}\phi_{0,2})+8\Delta\phi_{-2,1}^3\\ 
=&\frac{1}{4}(E_{6,2}\phi_{0,1}-E_6\phi_{0,3})+6\Delta\phi_{-2,1}^3=2E_{6,3}+\frac{288}{61}\Delta\phi_{-2,1}^3,
\end{align*}
and
$$J^{w,\ZZ}_{2*,*}\subsetneq\ZZ\left[\frac{1}{2},\frac{1}{3}\right] [E_4,E_6,\phi_{0,1},\phi_{-2,1}].$$
\end{theorem}

\begin{proof}We divide the proof into several steps.

\textbf{(i)} By Theorem \ref{th:Gri1} (c), we have 
$$J^{w,\ZZ}_{0,*}=\ZZ [\phi_{0,1},\phi_{0,2},\phi_{0,3},\phi_{0,4}].$$
From $J_{2*,*}^w=\CC[E_4,E_6,\phi_{0,1},\phi_{-2,1} ]$, we get 
$$ J^{w,\ZZ}_{-2k,m}=\phi_{-2,1}^k J^{w,\ZZ}_{0,m-k}, \quad k>0,\; m>0.$$

\textbf{(ii)} Using Corollary \ref{cor:1JF} we find 
Jacobi forms whose $q^0$-terms contain only coefficients
$1$ or $\zeta^{\pm i}$ for $1\leq i \leq m$
when $2k\geq 4$ and $m\geq 1$. We construct them in the ring
$$
\mathcal{A}=\ZZ [E_4,E_6, E_{4,1},E_{4,2},E_{4,3},
E_{6,1},E_{6,2},F_{6,3},\phi_{0,1},\phi_{0,2},\phi_{0,3},\phi_{0,4}].
$$
We put 
\begin{align*}
F_{2k,m}^{(1)}&= F_{2k,m-1}\phi_{0,1}=\zeta^{\pm 1}+10+O(q), \quad m\geq 1,\\
F_{2k,m}^{(i)}&= F_{2k,0}\psi_{0,m}^{(i)}=\zeta^{\pm i}+c_0 +\sum_{1\leq j<i}c(j)\zeta^{\pm j}+O(q), \quad m\geq 2,
\end{align*}
where Jacobi forms $\psi_{0,m}^{(i)}$ are introduced in Theorem 
\ref{th:Gri1} (a) and $c(j)\in \ZZ$ and $2\leq i \leq m$. 

\textbf{(iii)} 
For arbitrary $\phi \in J_{2k,m}^{w,\ZZ}$ with $k\geq 2$, there exists a weak Jacobi form $\psi \in \mathcal{A}$ such that 
$$
\phi-\psi \in J_{2k,m}^{w,\ZZ}(q) \quad \text{and} \quad 
(\phi-\psi)/\Delta\in J_{2k-12,m}^{w,\ZZ}.
$$ 
For arbitrary $\phi \in J_{2,m}^{w,\ZZ}$, we have $\phi \in \phi_{-2,1}
J_{4,m-1}^{w,\ZZ} $ because $\phi(\tau,0)=0$. 

\textbf{(iv)} By means of steps (i) and (iii), we complete the proof by 
induction on the weight $2k$. In order to prove these relations except the last 
one for $F_{6,3}$,  we only need to check that their $q^0$-terms are the 
same because the corresponding spaces of Jacobi forms are of dimension one. 
For the last one, we need to compare their $q^0$-terms and $q^1$-terms.
\end{proof}

\noindent
{\bf Remark.}
The above theorem was announced without a proof in \cite[Theorem 2.3]{G99b} 
where the function $E_{6,3}'=E_{6,3}+\frac{22}{61}\Delta\phi_{-2,1}^3$ 
was used instead of $F_{6,3}$. The last formula shows that 
the Jacobi--Eisenstein series $E_{6,3}$ has non-integral Fourier 
coefficients.
In fact, we have $F_{6,3}=E_{6,3}'+2\Delta\phi_{-2,1}^3$ and there is a 
beautiful formula
\begin{equation}
E_{6,3}'=\theta_{01}^6\vartheta_{01}^6
+\frac{3}{2}\theta_{01}^4\theta_{10}^2\vartheta_{01}^4\vartheta_{10}^2
-\frac{3}{2}\theta_{01}^2\theta_{10}^4\vartheta_{01}^2\vartheta_{10}^4
-\theta_{10}^6\vartheta_{10}^6,
\end{equation} 
(see \S 2.2).
Since $\frac{1}{2}\theta_{10}$ has integral Fourier coefficients, the 
function $E_{6,3}'$ also has integral Fourier coefficients. 
\smallskip

Using Theorem \ref{th:Gri2} and relations 
\eqref{eq:structure1} and \eqref{eq:structure2} we get

\begin{corollary}\label{cor:Gri2}
The bigraded ring  $J^{w,\ZZ}_{*,*/2}$  of weak Jacobi forms of integral weights and 
half-integral indices with integral Fourier coefficients 
is finitely generated over $\ZZ$ and equal to
$$
\ZZ\left[E_4,E_6,\Delta, E_{4,1},E_{4,2},E_{4,3},
E_{6,1},E_{6,2},F_{6,3},\phi_{0,1},\phi_{0,2},\phi_{0,\frac{3}{2}},
\phi_{0,4},\phi_{-1,\frac{1}{2}}\right].
$$
\end{corollary}

Note that 
$$
\phi_{0,\frac{3}{2}}^2=\phi_{0,3},\quad
\phi_{-1,\frac{1}{2}}^2=\phi_{-2,1},\quad
\phi_{0,\frac{3}{2}}\phi_{-1,\frac{1}{2}}=\phi_{-1,2}.
$$

\section{Integral weakly holomorphic Jacobi forms and Borcherds products}\label{Sec:4}
In this section, we study the graded ring of all weakly holomorphic Jacobi 
forms of weight $0$ with integral Fourier coefficients. By the well-known 
isomorphism between vector-valued modular forms and Jacobi forms (see 
\cite{CG13}) or \cite{EZ85}, the space $J_{0,m}^{!}$ of weakly holomorphic 
Jacobi forms of weight $0$ and index $m$ is isomorphic to the space of 
nearly holomorphic modular forms of weight $-\frac{1}{2}$ for the Weil 
representation associated to the discriminant form of the lattice $A_1(m)$, 
i.e. the lattice $\ZZ$ of rank $1$
with bilinear form $2mx^2$. By means of this 
isomorphism, as a variant of Borcherds products (see \cite{Bor98}) at a one-dimensional cusp of a modular variety,  
Gritsenko and Nikulin \cite{GN98} constructed a lifting from weakly 
holomorphic Jacobi forms of weight $0$ to meromorphic Siegel paramodular 
forms.
 
Let $\phi \in J_{0,m}^{!,\ZZ}$ and $\Gamma_m$ be the paramodular 
group of polarisation $m$. 
Then the Borcherds product of $\phi$, denoted by $
\Borch(\phi)$, is a meromorphic paramodular form for $\Gamma_m$ with a 
character (see \cite[Theorem 2.1]{GN98}). 
The following theorem gives an explicit description of $J_{0,m}^{!,\ZZ}$.

\begin{theorem}\label{th:weaklyJF}
The graded ring of all weakly holomorphic Jacobi forms of weight $0$ and 
integral index with integral Fourier coefficients is finitely generated 
over $\ZZ$. Moreover, we have
$$
J^{!,\ZZ}_{0,*}=\ZZ \left[j,\phi_{0,1},\phi_{0,2},\phi_{0,3},\phi_{0,4},
\frac{E_4^2E_{4,1}}{\Delta},\frac{E_4^2E_{4,2}}{\Delta},\frac{E_4^2E_{4,3}}
{\Delta}\right],
$$
where $j={E_4^3}/{\Delta}$. The generators $\phi_{0,1}$, $\phi_{0,2}$ and $
\phi_{0,3}$ are algebraically independent over $\CC$. The other generators 
satisfy the following relations:
\begin{align*}
&4\phi_{0,4}=\phi_{0,1}\phi_{0,3}-\phi_{0,2}^2,\\
&\left( \phi_{0,1}^2-24\phi_{0,2}\right)^3=j \left(-\phi_{0,1}^2 
\phi_{0,4}+9\phi_{0,1}\phi_{0,2}\phi_{0,3}-8\phi_{0,2}^3-27\phi_{0,3}^2  
\right),\\
&6\left(\frac{E_4^2E_{4,1}}{\Delta} \right)^2-j
\phi_{0,1}\left(\frac{E_4^2E_{4,1}}{\Delta} \right)+72j\phi_{0,1}
^2+j(j-1728)\phi_{0,2}=0,\\
&6\frac{E_4^2E_{4,2}}{\Delta}=\left(\frac{E_4^2E_{4,1}}{\Delta} \right)
\phi_{0,1}-j\phi_{0,2},\\
&2\frac{E_4^2E_{4,3}}{\Delta}=\left(\frac{E_4^2E_{4,1}}{\Delta} \right)
\phi_{0,2}-j\phi_{0,3}.
\end{align*}
\end{theorem}

\begin{proof}
Let $\phi\in J^{!,\ZZ}_{0,m}$. If $m=0$, the function $\phi$ is in fact a 
modular function with a possible pole at infinity, which implies that $\phi
\in \ZZ[j]$. When $m>0$, there exists a natural number $l$ such that 
$\Delta^l\phi$ is a weak Jacobi form by the definition of weakly 
holomorphic 
Jacobi forms. We denote the smallest one by $n$. As in the proof of Theorem 
\ref{th:Gri2}, we can construct the following weak Jacobi forms of weight 
$12$ and index $m$ with integral Fourier coefficients:
\begin{align*}
G_{12,0}&=E_4^3, \quad G_{12,1}=E_4^2E_{4,1}, \quad G_{12,2}=E_4^2E_{4,2}, 
\quad G_{12,3}=E_4^2E_{4,3},\\
G_{12,m}&=G_{12,m-3}\phi_{0,3}-G_{12,m-4}\phi_{0,4}=1+O(q), \quad m\geq 4,
\\
G_{12,m}^{(1)}&= G_{12,m-1}\phi_{0,1}=\zeta^{\pm 1}+10+O(q), \quad m\geq 
1,\\
G_{12,m}^{(i)}&= E_{4}^3\psi_{0,m}^{(i)}=\zeta^{\pm i}+c_0 +\sum_{1\leq 
j<i}c(j)\zeta^{\pm j}+O(q),\quad m\geq 2,
\end{align*}
where $c(j)\in \ZZ$ and $2\leq i \leq m$.
Therefore, there exist integers $x_i\in \ZZ$ such that the $q^0$-term of 
the function
$$
\Phi_n= \Delta^n\phi-x_0E_4^{3(n-1)}G_{12,m}-\sum_{i=1}^mx_iE_4^{3(n-1)}
G_{12,m}^{(i)}
$$
equals $0$. Thus, 
$$
{\Phi_n}/{\Delta} \in J_{12(n-1),m}^{w,\ZZ}
$$
and the smallest natural number $l$ such that $\Delta^l(\Phi_n/\Delta^n)\in 
J^{w,\ZZ}_{12*,m}$ is equal to $n-1$. Hence we complete the proof by 
induction on $n$.
\end{proof}
  
In the procedure of writing $\phi\in J_{0,m}^{!,\ZZ}$ in terms of the given generators, only the following Fourier coefficients are required to execute the algorithm:  
 $$
f(n,l): \quad \forall \; n\leq \left[\frac{m}{6}\right],  \forall\; l\in \ZZ.
$$
This fact follows from the above proof and Theorem \ref{th:Jacobi}.

To finish the paper, we consider the  Borcherds products related to the 
generators of $J^{!,\ZZ}_{0,*}$.
The Fourier coefficients of the first generator $j(\tau)$
determine the remarkable infinite product formula of
the difference $j(\tau_1)-j(\tau_2)$.
This formula was independently proved by Koike, Norton, and Zagier
in the 1980s. It was the first proto-example of the Borcherds automorphic 
product construction (see details in \cite{Bor92}).

As we mentioned in the introduction, the generators 
$\phi_{0,1}$, $\dots$, $\phi_{0,4}$ give the basic Siegel paramodular 
forms  $\Borch(\phi_{0,m})$
$$
\Delta_5\in S_5(\Gamma_1,\chi_2),\quad
\Delta_2\in S_2(\Gamma_2,\chi_4),\quad
\Delta_1\in S_1(\Gamma_3,\chi_6)
$$
and $\Delta_{1/2}\in M_{1/2}(\Gamma_{4},\chi_8)$
with the simplest possible divisor 
$\Gamma_m<(z=0)>$ where
$\Gamma_m$ is the corresponding paramodular group acting on the Siegel 
upper half-plane $\HH_2$ and $\chi_l$
is a character (or a multiplier system for $\chi_8$) 
of order $l$ (see \cite{GN98} for details).

To describe the last three generators, we  consider
$$
\phi(\tau,z)=\sum f(n,l)q^n\zeta^l \in J_{0,m}^{!,\ZZ}.
$$ 
It is known that the Fourier coefficients $f(n,l)$ depend only on the 
number $4nm-l^2$ and the class of $l$ in $\ZZ/2m\ZZ$. In addition, 
$f(n,l)=f(n,-l)$. The non-zero Fourier coefficient $f(n,l)$ satisfying $4nm-l^2<0$ is called \textit{singular}.  The divisor of $\Borch(\phi)$ is the union of a finite 
number of Humbert modular surfaces and the multiplicities of irreducible 
Humbert modular surfaces are determined by the singular Fourier 
coefficients of $\phi$ (see \cite[Theorem 2.1]{GN98}).
Thus, for $1\leq m \leq 4$, the singular Fourier coefficients of $\phi$ 
are completely determined by the Fourier coefficients $f(n,l)$ with 
$n\leq 0$. By Theorem \ref{th:weaklyJF}, it is possible and easy to 
construct the Jacobi form with any given integral singular Fourier 
coefficients. 
Therefore, we can construct any paramodular form of level $\leq 4$ whose 
divisor is a certain linear combination of Humbert modular surfaces.

For the last three generators in Theorem \ref{th:weaklyJF}
the singular part has a very simple form. We have 
$$
\psi_{0,1}=\frac{E_4^2E_{4,1}}{\Delta}-56\phi_{0,1}=q^{-1}+
\zeta^{\pm2}+70+O(q)\in J_{0,1}^{!,\ZZ}.
$$
The Borcherds product of $\psi_{0,1}$ is the basic Siegel modular form of 
odd weight, namely the Igusa modular form $\Delta_{35}$. 
Similarly, we have
\begin{align*}
\psi_{0,2}&=\frac{E_4^2E_{4,2}}{\Delta}-14\phi_{0,1}^2+216\phi_{0,2}
=q^{-1}+24+O(q)\in J_{0,2}^{!,\ZZ},\\
\psi_{0,3}&=\frac{E_4^2E_{4,3}}{\Delta}
-2\phi_{0,1}^3+33\phi_{0,1}\phi_{0,2}+90\phi_{0,3}
=q^{-1}+24+O(q)\in J_{0,3}^{!,\ZZ}.
\end{align*}
Then   
$\Borch(\psi_{0,2})$ and $\Borch(\psi_{0,3})$ are strongly reflective 
modular forms of weight $12$ with complete $2$-divisor for $\Gamma_2$ and $
\Gamma_3$, respectively. 

In the theory of the Lorentzian Kac--Moody algebras
the paramodular forms  $\Delta_5$, $\Delta_2$, $\Delta_1$  and 
$\Delta_{35}$ determine four basic algebras with 
a Weyl vector of negative norm
(in the signature $(2,1)$) and the forms $\Delta_{1/2}$,
$\Borch(\psi_{0,2})$ and $\Borch(\psi_{0,3})$ determine Kac--Moody algebras with 
a Weyl vector of norm zero.
In particular, $\Delta_{35}$ determines the Lorentzian Kac--Moody algebra 
with the simplest hyperbolic group isomorphic to $\hbox{GL}_2(\ZZ)$ and 
the smallest Cartan matrix (see \cite{GN98} for details). 

Let us remark that the results of the paper give some further applications 
to the theory of reflective modular forms. 
There is an identity involving singular Fourier coefficients 
of a weakly holomorphic Jacobi form of weight $0$ (see \cite[Lemma 2.2]
{GN98}):
\begin{equation}\label{eq:q^0}
\sum_{l\in \ZZ}f(0,l)-\frac{6}{m} \sum_{l\in\ZZ}f(0,l)l^2-24\sum_{n<0,l\in
\ZZ}f(n,l)\sigma_1(-n)=0.
\end{equation}
For any given Borcherds product $F$, if the multiplicities of its divisor 
are known, then we know the singular Fourier coefficients of the 
corresponding weakly holomorphic Jacobi form $\phi$ of weight $0$. Applying 
\eqref{eq:q^0} to $\phi$,  we can calculate the weight of $F$ as in 
\cite{Wan18} for the case of reflective modular forms.

\bigskip

\noindent
\textbf{Acknowledgements.}
The work is supported by the Laboratory of Mirror Symmetry NRU HSE (RF government grant, ag. N 14.641.31.0001). H. Wang is grateful to Max Planck Institute for Mathematics in Bonn for its hospitality and financial support. The authors thank the referee for helpful comments and suggestions that make this paper better.

\bibliographystyle{amsplain}

\end{document}